\documentclass[a4paper,12pt]{article}

\usepackage[utf8]{inputenc}
\usepackage[english]{babel}
\usepackage{amsmath, amssymb, amsthm}
\usepackage{enumerate}
\usepackage[colorlinks=true,linkcolor=blue,citecolor=blue]{hyperref}

\usepackage{tikz}
\usetikzlibrary{fadings}
\tikzfading[name=fade out,
            inner color=transparent!0,
            outer color=transparent!100]
\usetikzlibrary{calc,arrows,decorations.pathreplacing,fadings,3d,positioning}

\title{Problems concerning \\ Diophantine exponents of lattices.
       \thanks{ This research was supported in part by RSF grant 14-11-00433}}
\author{Oleg\,N.\,German}
\date{}

\theoremstyle{definition}
\newtheorem{definition}{Definition}

\newtheorem*{notation*}{Notation}

\theoremstyle{remark}

\newtheorem*{remark*}{Remark}

\theoremstyle{plain}
\newtheorem{theorem}{Theorem}

\newtheorem{proposition}{Proposition}

\newtheorem{problem}{Problem}

\newtheorem*{statement*}{Statement}
\newtheorem*{corollary*}{Corollary}

\DeclareMathOperator{\conv}{conv}

\renewcommand{\phi}{\varphi}

\renewcommand{\vec}[1]{\mathbf{#1}}
\renewcommand{\geq}{\geqslant}
\renewcommand{\leq}{\leqslant}

\newcommand{\e}{\varepsilon}

\newcommand{\R}{\mathbb{R}}
\newcommand{\Z}{\mathbb{Z}}
\newcommand{\Q}{\mathbb{Q}}
\newcommand{\N}{\mathbb{N}}

\newcommand{\Complex}{\mathbb{C}}

\newcommand{\La}{\Lambda}

\newcommand{\bpsi}{\underline{\psi}}
\newcommand{\apsi}{\overline{\psi}}
\newcommand{\bPsi}{\underline{\Psi}}
\newcommand{\aPsi}{\overline{\Psi}}

\newcommand{\cB}{\mathcal{B}}
\newcommand{\cC}{\mathcal{C}}
\newcommand{\cD}{\mathcal{D}}

\newcommand{\cF}{\mathcal{F}}

\newcommand{\cK}{\mathcal{K}}
\newcommand{\cL}{\mathcal{L}}

\newcommand{\cO}{\mathcal{O}}

\newcommand{\cV}{\mathcal{V}}

\newcommand{\SL}{\textup{SL}}

\newcommand{\starv}{\textup{St}_{\vec v}}

\begin{document}

\maketitle

\begin{abstract}
  In this paper we give a survey of what is currently known about Diophantine exponents of lattices and propose several problems.
\end{abstract}


\section{Introduction}

Let $L_1,\ldots,L_n$ be $n$ linearly independent linear forms in $\R^d$, $n\leq d$. One of the basic questions in Diophantine approximation is how small the $n$-tuple
\[\big(L_1(\vec z),\ldots,L_n(\vec z)\big)\]
can be for large $\vec z\in\Z^d$. There are two classical ways to measure the ``size'' of this $n$-tuple. The first one is to consider an arbitrary norm, say, the sup-norm, and the second one is to consider the product of the absolute values of the entries. Then, the question is how fast this quantity can tend to zero with the growth of the ``size'' of $\vec z$.

For $n=1$ the ``norm'' approach gives us the problem of approximating zero with the values of a given linear form at integer points. For $n=d-1$ it leads to the dual problem of simultaneous approximation of $d-1$ real numbers with rationals having same denominator.

The multiplicative approach leads to a variety of more complicated problems, of which probably the most famous one is the Littlewood conjecture, which claims that, given two forms $L_1$, $L_2$ in three variables with coefficients written in the rows of
\[ \begin{pmatrix}
     \theta_1 & 1 & 0      \\
     \theta_2 & 0 & 1
   \end{pmatrix}, \]
for each $\e>0$ the inequality
\[ \prod_{i=1,2}|L_i(\vec z)|\leq\e z_1^{-1} \]
admits infinitely many solutions in $\vec z=(z_1,z_2,z_3)\in\Z^3$, $z_1\neq0$.

In this paper we are interested in the case $n=d$. Then, if $\vec z$ is large, the ``norm'' approach cannot give small values of $\big(L_1(\vec z),\ldots,L_d(\vec z)\big)$, so, we cannot properly talk about Diophantine approximation in this sense. But the multiplicative approach brings us to an area that is rich with very natural unsolved problems concerning dynamics in the space of unimodular lattices. Since $L_1,\ldots,L_d$ are assumed to be linearly independent, the set
\begin{equation} \label{eq:lattice}
  \La=\Big\{\big(L_1(\vec z),\ldots,L_d(\vec z)\big)\,\Big|\,\vec z\in\Z^d \Big\}
\end{equation}
is a lattice of rank $d$. Let us set for each $\vec x=(x_1,\ldots,x_d)\in\R^d$
\[\Pi(\vec x)=\prod_{1\leq i\leq d}|x_i|^{1/d}.\]
The main purpose of this paper is to give a survey of what is currently known about how fast $\Pi(\vec x)$ can decay as $\vec x$ ranges through the points of $\La$, and to formulate some questions that need answering.

Probably the simplest quantity that characterizes the asymptotic behaviour of $\Pi(\vec x)$ is the \emph{Diophantine exponent} of $\La$.

\begin{definition} \label{d:lattice_exponent}
  Let $\La$ be a lattice of full rank in $\R^d$.
  The \emph{Diophantine exponent} of $\La$ is defined as
  \[\omega(\La)=\sup\Big\{\gamma\in\R\ \Big|\,\Pi(\vec x)\leq|\vec x|^{-\gamma}\text{ for infinitely many }\vec x\in\La \Big\},\]
  where $|\cdot|$ is the sup-norm.
\end{definition}

In other words,
\begin{equation} \label{eq:omega_as_limsup}
  \omega(\La)=
  \limsup_{\substack{\vec v\in\La \\ |\vec v|\to\infty}}\frac{\log\big(\Pi(\vec v)^{-1}\big)}{\log|\vec v|}\,.
\end{equation}
Clearly, Definition \ref{d:lattice_exponent} does not depend on the choice of the norm.

It follows from Minkowski's convex body theorem that
\[\omega(\La)\geq0.\]
This trivial bound is sharp. For instance, Schmidt's subspace theorem provides a rich family of lattices having zero exponent. The following statement can be found in \cite{skriganov_1998}, \cite{german_2017}.

\begin{theorem}[Corollary to Schmidt's subspace theorem] \label{t:schmidt_finitely_many_points}
  Let $L_1(\vec z),\ldots,L_d(\vec z)$ be linearly independent linear forms in $d$ variables with algebraic coefficients. Suppose that for each $k$-tuple $(i_1,\ldots,i_k)$, $1\leq i_1<\ldots<i_k\leq d$, $1\leq k\leq d$,
  the coefficients of the multivector
  \[ L_{i_1}\wedge\ldots\wedge L_{i_k} \]
  are linearly independent over $\Q$.
  Then for each $\e>0$ there are only finitely many points $\vec z\in\Z^d$ satisfying
  \[ \prod_{1\leq i\leq d}\big|L_i(\vec z)\big|<|\vec z|^{-\e}. \]
\end{theorem}

Thus, in the case of algebraic coefficients satisfying the independence condition mentioned in Theorem \ref{t:schmidt_finitely_many_points} we have $\omega(\La)=0$ for $\La$ defined by \eqref{eq:lattice}. It appears that, same as with real numbers, such an algebraic lattice behaves as an average unimodular lattice. Denote by $\cL_d$ the space of full rank lattices in $\R^d$ of covolume $1$,
\[\cL_d\cong\SL_d(\R)/\SL_d(\Z).\]
It was shown by Skriganov \cite{skriganov_1998} that for almost every $\La\in\cL_d$ we have
\[\Pi(\vec x)^d\gg_{\La,\e}(\log(1+|\vec x|))^{1-d-\e}\quad\text{ for }\vec x\in\La\backslash\{\vec 0\}.\]
Thus, for almost every $\La\in\cL_d$ we have $\omega(\La)=0$. Later on D.\,Kleinbock and G.\,Margulis \cite{kleinbock_margulis_1999} completed Skriganov's theorem to a proper multidimensional multiplicative generalization of Khintchine’s theorem.


\begin{theorem}[Kleinbock, Margulis, 1999] \label{t:kleinbock_margulis}
  Let ${f}:[1,+\infty)\to(0,+\infty)$ be a non-increasing continuous function. Then for almost every (resp. almost no) $\La\in\cL_d$ there are infinitely many $\vec x\in\La$ such that
  \[\Pi(\vec x)^d\leq|\vec x|\cdot{f}(|\vec x|),\]
  provided the integral
  \[\int_1^\infty(\log x)^{d-2}{f}(x)dx\]
  diverges (resp. converges).
\end{theorem}

Among the lattices with zero Diophantine exponent there is a class of lattices that can be viewed as a multidimensional multiplicative generalization of badly approximable numbers. Those are the lattices with positive \emph{norm minimum}
\[N(\La)=\inf_{\vec x\in\La\backslash\{\vec 0\}}\Pi(\vec x)^d.\]
According to Mahler's compactness criterion (see \cite{cassels_GN}) $N(\La)$ is positive if and only if the orbit $\cD_d\La$ of $\La$ under the action of the group of diagonal matrices
\begin{equation} \label{eq:diagonal_matrices}
  \cD_d=\Big\{\textup{diag}(e^{t_1},\ldots,e^{t_d})\,\Big|\,t_1,\ldots,t_d\in\R,\ \sum_{i=1}^dt_i=0 \Big\}
\end{equation}
is relatively compact. From this point of view $\omega(\La)$ is responsible for how fast $\La$ can leave any given compact set under the action of $\cD_d$.

It is an intriguing question which lattices have positive norm minimum. For instance, if $E$ is a totally real algebraic extension of $\Q$ of degree $d$, $\sigma_1,\ldots,\sigma_d$ are its embeddings into $\Complex$ (actually into $\R$), and $\theta_1,\ldots,\theta_d$ is a basis of $E$ over $\Q$, then it can be easily shown with the help of Liouville-type argument that $N(\La)>0$ for
\[\La=
  \begin{pmatrix}
    \sigma_1(\theta_1) & \sigma_1(\theta_2) & \cdots & \sigma_1(\theta_d) \\
    \sigma_2(\theta_1) & \sigma_2(\theta_2) & \cdots & \sigma_2(\theta_d) \\
    \vdots             & \vdots             & \ddots & \vdots             \\
    \sigma_d(\theta_1) & \sigma_d(\theta_2) & \cdots & \sigma_d(\theta_d)
  \end{pmatrix}\Z^d.\]
It is a famous conjecture, sometimes referred to as Margulis--Cassels--Swinnerton-Dyer conjecture, that for $d\geq3$, up to homotheties and the action of $\cD_d$, these are the only lattices with positive norm minimum. Dynamically, it states that for $d\geq3$ an orbit $\cD_d\La$ is compact if and only if it is relatively compact (see \cite{margulis_2000}, \cite{cassels_swinnerton_dyer}). It is well known (see \cite{cassels_swinnerton_dyer}) that the three-dimensional Margulis--Cassels--Swinnerton-Dyer conjecture implies the Littlewood conjecture, and that the set of counterexamples has zero Hausdorff dimension (see \cite{einsiedler_katok_lindenstrauss}).

The rest of the paper is organized as follows. Section \ref{sec:2dim} is devoted to the two-dimensional case. We show that in this case lattice exponents are closely connected to the measure of irrationality of a real number, which can be easily dealt with due to regular continued fractions, a powerful tool that is available when $d=2$.
In Section \ref{sec:spectra} we are concerned about the set of values lattice exponents can attain. In Section \ref{sec:multicontractions} we generalize to the multidimensional case the connection between the measure of irrationality and the growth of partial quotients, taking Klein polyhedra as a multidimensional analogue of continued fractions. Finally, in Section \ref{sec:parametric} we discuss a possible way to generalize Schmidt--Summerer's parametric geometry of numbers to our setting.

\section{Two-dimensional background} \label{sec:2dim}

\subsection{Lattice exponents and irrationality measure} \label{sec:exponents_vs_irrationality_measure}

Let $\theta_1$, $\theta_2$ be distinct real numbers. Consider linear forms $L_1$, $L_2$ in two variables with coefficients written in the rows of
\begin{equation} \label{eq:A_theta}
  A=
  \begin{pmatrix}
    \theta_1 & -1 \\
    \theta_2 & -1
  \end{pmatrix}
\end{equation}
and set
\begin{equation} \label{eq:La_theta}
  \La_{\theta_1,\theta_2}=A\Z^2=
  \Big\{\big(L_1(\vec z),L_2(\vec z)\big)\,\Big|\,\vec z\in\Z^2 \Big\}.
\end{equation}
Then for each $\,\vec x=\big(L_1(\vec z),L_2(\vec z)\big)\in\La$, where $\ \vec z=(q,p)\in\Z^2$, we have
\begin{equation} \label{eq:Pi_vs_q_and_p}
  \Pi(\vec x)^2=|L_1(\vec z)|\cdot|L_2(\vec z)|=|q\theta_1-p|\cdot|q\theta_2-p|
\end{equation}
and
\begin{equation} \label{eq:x_z_q_are_same}
  |\vec x|\asymp|L_i(\vec z)|\asymp|\vec z|\asymp|q|
  \ \text{ whenever }\ |L_j(\vec z)|\leq1,\ i\neq j.
\end{equation}
Hence
\begin{equation} \label{eq:omega_vs_mu}
  \omega(\La_{\theta_1,\theta_2})=\frac12\max\big(\mu(\theta_1),\mu(\theta_2)\big)-1,
\end{equation}
where
\[\mu(\theta)=\sup\Big\{\gamma\in\R\ \Big|\,\big|\theta-p/q\big|\leq|q|^{-\gamma}\text{ admits $\infty$ solutions in }(q,p)\in\Z^2 \Big\}\]
is the \emph{measure of irrationality} of a number $\theta$.

Thus, in the two-dimensional case lattice exponents simply provide another viewpoint at irrationality measure. Particularly, results concerning irrationality measure can be reformulated in terms of lattice exponents. For instance, 
the Jarn\'ik--Besicovitch theorem turns into

\begin{theorem}[Reformulation of the Jarn\'ik--Besicovitch theorem] \label{t:jarnik_besicovitch_for_lattices}
  \[\dim_\textup{H}\Big\{\Lambda\in\mathcal L_2\,\Big|\,\omega(\Lambda)\geq\lambda \Big\}=\frac{\lambda+2}{\lambda+1}\,.\]
\end{theorem}

It is also worth mentioning the connection between the norm minimum of $\La_{\theta_1,\theta_2}$ and the property of $\theta_1,\theta_2$ to be badly approximable. It follows from \eqref{eq:Pi_vs_q_and_p}, \eqref{eq:x_z_q_are_same}.

\begin{proposition} \label{prop:norm_minimum_vs_badly_approximable}
  $N(\La_{\theta_1,\theta_2})>0$ if and only if both $\theta_1$ and $\theta_2$ are badly approximable.
\end{proposition}

\subsection{Irrationality measure and growth of partial quotients}

Given a real number $\theta$ and its continued fraction expansion $\theta=[a_0;a_1,a_2,\ldots]$, let $p_n/q_n$ denote its $n$-th convergent. Then, as is well known,
\begin{equation} \label{eq:mu_vs_partial_quotients}
  \mu(\theta)=2+\limsup_{n\to\infty}\frac{\log a_{n+1}}{\log q_n}\,.
\end{equation}
Thus, in the two-dimensional case, knowing \eqref{eq:omega_vs_mu} and \eqref{eq:mu_vs_partial_quotients}, we can easily construct lattices with any given nonnegative Diophantine exponent. For instance, given $\omega\geq0$, we can take as $\theta_1$ a badly approximable number, i.e. a number with bounded partial quotients, and define $\theta_2=[a_0;a_1,a_2,\ldots]$, say, by the recurrence relation
\[a_{n+1}=\big[q_n^{2\omega}\big].\]
Then by \eqref{eq:omega_vs_mu} and \eqref{eq:mu_vs_partial_quotients} we have $\omega(\La_{\theta_1,\theta_2})=\omega$.

\subsection{Klein polygons} \label{sec:klein_polygons}

There is a nice geometric interpretation of \eqref{eq:mu_vs_partial_quotients} in terms of lattice exponents and \emph{Klein polygons}.

Let $\theta_1$, $\theta_2$ be real numbers, as in Section \ref{sec:exponents_vs_irrationality_measure}, and let $L_1$, $L_2$, $\La_{\theta_1,\theta_2}$ be defined by \eqref{eq:A_theta} and \eqref{eq:La_theta}. Suppose for simplicity $\theta_1>1$, $-1<\theta_2<0$. This assumption slightly affects generality, but it helps to see the essence more clearly. Set
\[\cK_1=\conv\Big(\Big\{\vec z\in\Z^2\ \Big|\,L_1(\vec z)>0,\ L_2(\vec z)<0 \Big\}\Big),\]
\[\cK_2=\conv\Big(\Big\{\vec z\in\Z^2\ \Big|\,L_1(\vec z)<0,\ L_2(\vec z)<0 \Big\}\Big).\]
These convex hulls (see Fig.\ref{fig:KP_and_CF}) are called \emph{Klein polygons}.

\begin{figure}[h]
  \centering
  \begin{tikzpicture}[scale=1.4]
    \draw[very thin,color=gray,scale=1] (-3.8,-1.7) grid (4.8,55/8-0.2);

    \draw (1+0.4,-0.16) node[left,circle,inner sep=7pt,inner color=white,path fading=fade out]{$\vec v_{-2}$};
    \draw (-0.21,1-0.17) node[right,circle,inner sep=7pt,inner color=white,path fading=fade out]{$\vec v_{-1}$};
    \draw (-2.21,1-0.17) node[right,circle,inner sep=7pt,inner color=white,path fading=fade out]{$\vec v_{-3}$};

    \draw[color=black] plot[domain=-15/11:5] (\x, {11*\x/8}) node[above]{$y=\theta_1x$};
    \draw[color=black] plot[domain=-4:5] (\x, {-3*\x/8}) node[below]{$y=\theta_2x$};

    \draw (3+0.4,-1.16) node[left,circle,inner sep=7pt,inner color=white,path fading=fade out]{$\vec v_{-4}$};
    \draw (1+0.32,1.1) node[left,circle,inner sep=10pt,inner color=white,path fading=fade out]{$\vec v_0$};

    \fill[blue!10!,path fading=east]
        (3+1.8,4+1.8*7/5) -- (3,4) -- (1,1) -- (1,0) -- (3,-1) -- (3+1.8,-1-1.8*2/5) -- cycle;
    \fill[blue!10!,path fading=north]
        (2+2.8,3+2.8*4/3) -- (2,3) -- (0,1) -- (-2-1.8,3+2.8*4/3) -- cycle;
    \fill[blue!10!,path fading=west]
        (-2-1.8,3+2.8*4/3) -- (0,1) -- (-2,1) -- (-2-1.8,1+1.8/3) -- cycle;
    \fill[blue!10!,path fading=south]
        (-0.7,-1.7) -- (0,-1) -- (2,-1) -- (4.1,-1.7) -- cycle;
    \fill[blue!10!,path fading=west]
        (-3.8,-1.7) -- (-1,0) -- (-3,1) -- (-3.8,1+0.8*2/5) -- cycle;
    \fill[blue!10!,path fading=south]
        (-1-0.7*2/3,-1.7) -- (-1,-1) -- (-1,0) -- (-3.8,-1.7) -- cycle;

    \draw[color=blue] (3+1.8,4+1.8*7/5) -- (3,4) -- (1,1) -- (1,0) -- (3,-1) -- (3+1.8,-1-1.8*2/5);
    \draw[color=blue] (2+2.8,3+2.8*4/3) -- (2,3) -- (0,1) -- (-2,1) -- (-2-1.8,1+1.8/3);
    \draw[color=blue] (-0.7,-1.7) -- (0,-1) -- (2,-1) -- (4.1,-1.7);
    \draw[color=blue] (-1-0.7*2/3,-1.7) -- (-1,-1) -- (-1,0) -- (-3,1) -- (-3.8,1+0.8*2/5);

    \node[left,circle,inner sep=2pt,inner color=white] at (3-0.1,4.1) {};
    \node[circle,inner sep=3pt,inner color=white] at (2.19,3-0.11) {};
    \node[left,circle,inner sep=7pt,inner color=white,path fading=fade out] at (3+0.2,4.1) {$\vec v_2$};
    \node[right,circle,inner sep=7pt,inner color=white,path fading=fade out] at (2-0.22,3-0.13) {$\vec v_1$};

    \node[fill=blue,circle,inner sep=1.2pt] at (3,4) {};
    \node[fill=blue,circle,inner sep=1.2pt] at (1,1) {};
    \node[fill=blue,circle,inner sep=1.2pt] at (1,0) {};
    \node[fill=blue,circle,inner sep=1.2pt] at (3,-1) {};
    \node[fill=blue,circle,inner sep=1.2pt] at (2,3) {};
    \node[fill=blue,circle,inner sep=1.2pt] at (0,1) {};
    \node[fill=blue,circle,inner sep=1.2pt] at (-2,1) {};

    \node[fill=blue,circle,inner sep=1.2pt] at (-1,-1) {};
    \node[fill=blue,circle,inner sep=1.2pt] at (-1,0) {};
    \node[fill=blue,circle,inner sep=1.2pt] at (-3,1) {};
    \node[fill=blue,circle,inner sep=1.2pt] at (0,-1) {};
    \node[fill=blue,circle,inner sep=1.2pt] at (2,-1) {};

    \node[fill=blue,circle,inner sep=0.8pt] at (1,-1) {};
    \node[fill=blue,circle,inner sep=0.8pt] at (-1,1) {};
    \node[fill=blue,circle,inner sep=0.8pt] at (1,2) {};

    \node[right] at (1-0.03,0.5) {$a_0$}; 
    \node[right] at (2-2/13,2.5-3/13) {$a_2$};
    \node[right] at (3+1.2,4+1.2*7/5) {$a_4$};
    \node[right] at (2-0.06,-0.4) {$a_{-2}$};

    \node[above left] at (1.08,2-0.02) {$a_1$};
    \node[left] at (2+1.5*1.07,3+2*1.07) {$a_3$};
    \node[above] at (-1,0.95) {$a_{-1}$};

    \draw[blue] ([shift=({atan(-1/2)}:0.2)]1,0) arc (atan(-1/2):90:0.2);
    \draw[blue] ([shift=({atan(-2/5)}:0.2)]3,-1) arc (atan(-2/5):90+atan(2):0.2);
    \draw[blue] ([shift=(-90:0.2)]1,1) arc (-90:atan(3/2):0.2);
    \draw[blue] ([shift=({-90-atan(2/3)}:0.2)]3,4) arc (-90-atan(2/3):atan(7/5):0.2);
    \draw[blue] ([shift=({atan(4/3)}:0.2)]2,3) arc (atan(4/3):225:0.2);
    \draw[blue] ([shift=(45:0.2)]0,1) arc (45:180:0.2);
    \draw[blue] ([shift=(0:0.2)]-2,1) arc (0:90+atan(3):0.2);

    \node[right] at (3.15,-0.88) {$a_{-3}$};
    \node[right] at (1.17,0) {$a_{-1}$};
    \node[right] at (1.17,1) {$a_1$};
    \node[right] at (3.08,3.75) {$a_3$};
    \node[left] at (2-0.1,3.23) {$a_2$};
    \node[above] at (-0.05,1.18) {$a_0$};
    \node[above] at (-2.02,1.16) {$a_{-2}$};

    \draw (4.26,1.5) node[right]{$\cK_1$};
    \draw (-0.74,4.5) node[right]{$\cK_2$};
  \end{tikzpicture}
  \caption{Klein polygons and continued fractions} \label{fig:KP_and_CF}
\end{figure}
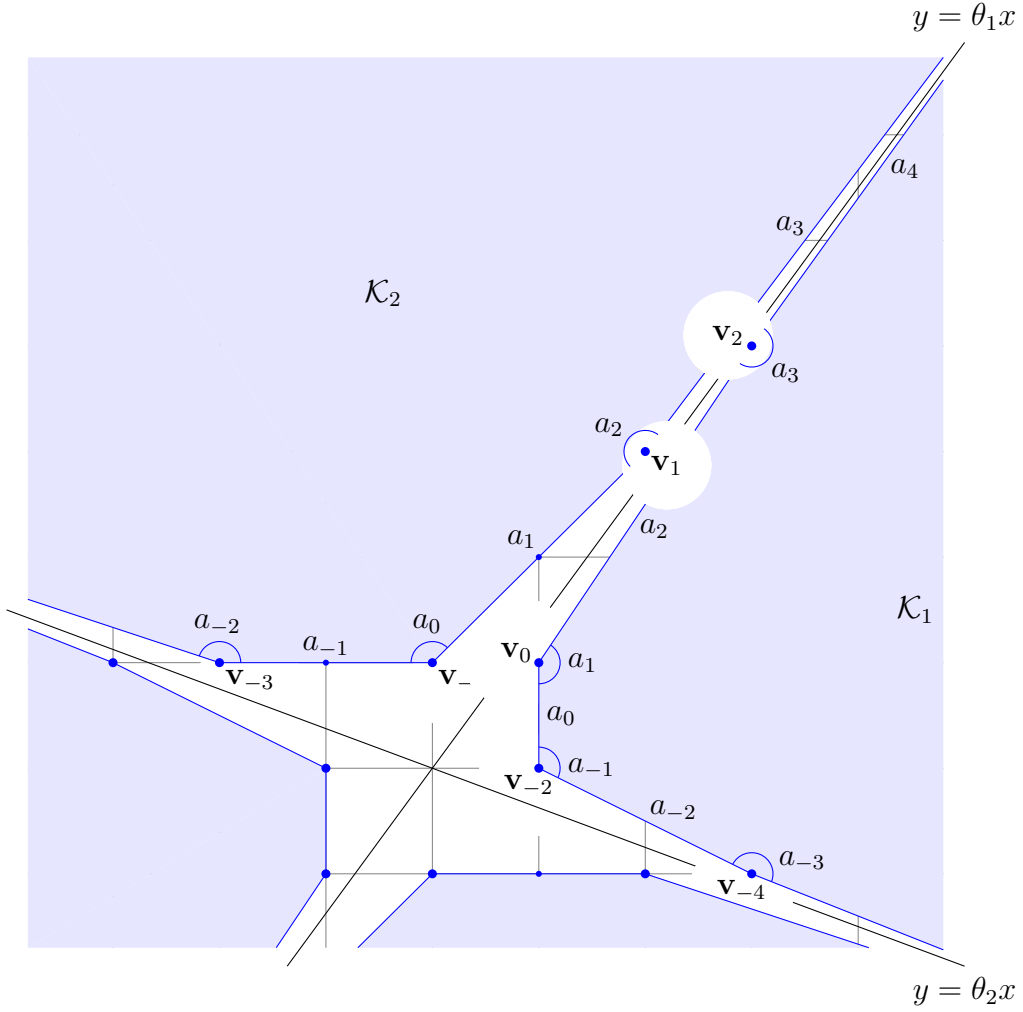

The integer-combinatorial structure of he boundaries $\partial\cK_1$ and $\partial\cK_2$ is closely connected to the continued fractions of $\theta_1$ and $\theta_2$. A detailed exposition of this connection can be found, for instance, in \cite{karpenkov_book} or \cite{german_tlyustangelov}. Here we shall confine ourselves to mentioning that the vertices of $\cK_1$ and $\cK_2$ have coordinates equal to the denominators and numerators of convergents of $\theta_1$ and $\theta_2$, and the integer lengths of their edges equal the corresponding partial quotients of $\theta_1$ and $\theta_2$. We remind that the \emph{integer length} of an integer segment (i.e. a segment with integer endpoints) is the number of empty integer subsegments contained in it. Moreover, in the same way partial quotients are ``attached'' to the edges of $\cK_1$ and $\cK_2$, they can also be ``attached'' to the vertices of $\cK_1$ and $\cK_2$. The reason for this is illustrated by Fig.\ref{fig:edge_vs_sprout}. More precisely, there is a bijection between the vertices of $\cK_1$ and the edges of $\cK_2$ such that a vertex $\vec v$ corresponds to an edge whose integer length is equal to
\[\alpha(\vec v)=|\det(\vec r_1,\vec r_2)|,\]
where $\vec r_1$ and $\vec r_2$ are primitive integer vectors parallel to the edges incident to $\vec v$. At Fig.\ref{fig:edge_vs_sprout} we have $\vec r_1=\vec w-\vec v$, $\vec r_2=\vec u-\vec v$. The quantity $\alpha(\vec v)$ is referred to as \emph{integer angle} at $\vec v$.

Thus, Klein polygons equipped with integer lengths of edges and integer angles at vertices can be viewed as a geometric interpretation of continued fractions.

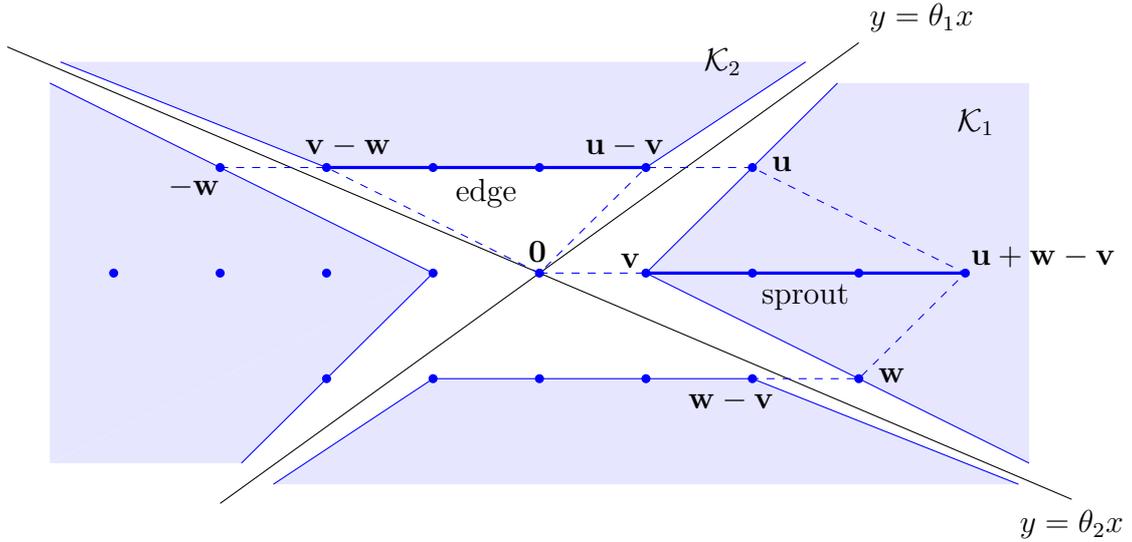
\begin{figure}[h]
  \centering
  \begin{tikzpicture}[scale=1.4]
    \draw[color=black] plot[domain=-3:3] (\x, {8*\x/11}) node[above right]{$y=\theta_1x$};
    \draw[color=black] plot[domain=-5:5] (\x, {-3*\x/7}) node[below]{$y=\theta_2x$};

    \fill[blue!10!,path fading=east]
        (4.6,1.8) -- (1,0) -- (4.6,-1.8) -- cycle;
    \fill[blue!10!,path fading=north]
        (2.8,1.8) -- (1,0) -- (4.6,1.8) -- cycle;
    \fill[blue!10!,path fading=west]
        (-4.6,-1.8) -- (-1,0) -- (-4.6,1.8) -- cycle;
    \fill[blue!10!,path fading=south]
        (-2.8,-1.8) -- (-1,0) -- (-4.6,-1.8) -- cycle;
    \fill[blue!10!,path fading=north]
        (-4.5,2) -- (-2,1) -- (1,1) -- (2.5,2) -- cycle;
    \fill[blue!10!,path fading=south]
        (4.5,-2) -- (2,-1) -- (-1,-1) -- (-2.5,-2) -- cycle;

    \draw[color=blue] (2.8,1.8) -- (1,0) -- (4.6,-1.8);
    \draw[color=blue] (-2.8,-1.8) -- (-1,0) -- (-4.6,1.8);
    \draw[color=blue] (-4.5,2) -- (-2,1) -- (1,1) -- (2.5,2);
    \draw[color=blue] (4.5,-2) -- (2,-1) -- (-1,-1) -- (-2.5,-2);

    \draw[very thick,color=blue] (-2,1) -- (1,1);
    \draw[very thick,color=blue] (1,0) -- (4,0);

    \draw[dashed,color=blue] (-3,1) -- (-2,1) -- (0,0) -- (1,1) -- (2,1) -- (4,0) -- (3,-1) -- (2,-1);
    \draw[dashed,color=blue] (0,0) -- (1,0);

    \node[fill=blue,circle,inner sep=1.2pt] at (-4,0) {};
    \node[fill=blue,circle,inner sep=1.2pt] at (-3,0) {};
    \node[fill=blue,circle,inner sep=1.2pt] at (-2,0) {};
    \node[fill=blue,circle,inner sep=1.2pt] at (-1,0) {};
    \node[fill=blue,circle,inner sep=1.2pt] at (0,0) {};
    \node[fill=blue,circle,inner sep=1.2pt] at (1,0) {};
    \node[fill=blue,circle,inner sep=1.2pt] at (2,0) {};
    \node[fill=blue,circle,inner sep=1.2pt] at (3,0) {};
    \node[fill=blue,circle,inner sep=1.2pt] at (4,0) {};

    \node[fill=blue,circle,inner sep=1.2pt] at (-3,1) {};
    \node[fill=blue,circle,inner sep=1.2pt] at (-2,1) {};
    \node[fill=blue,circle,inner sep=1.2pt] at (-1,1) {};
    \node[fill=blue,circle,inner sep=1.2pt] at (0,1) {};
    \node[fill=blue,circle,inner sep=1.2pt] at (1,1) {};
    \node[fill=blue,circle,inner sep=1.2pt] at (2,1) {};

    \node[fill=blue,circle,inner sep=1.2pt] at (-2,-1) {};
    \node[fill=blue,circle,inner sep=1.2pt] at (-1,-1) {};
    \node[fill=blue,circle,inner sep=1.2pt] at (0,-1) {};
    \node[fill=blue,circle,inner sep=1.2pt] at (1,-1) {};
    \node[fill=blue,circle,inner sep=1.2pt] at (2,-1) {};
    \node[fill=blue,circle,inner sep=1.2pt] at (3,-1) {};

    \draw (-0.02,0) node[above]{$\vec 0$};
    \draw (1.05,-0.05) node[above left]{$\vec v$};
    \draw (2.08,1.03) node[right]{$\vec u$};
    \draw (3.08,-1+0.03) node[right]{$\vec w$};
    \draw (4-0.05,-0.05) node[above right]{$\vec u+\vec w-\vec v$};
    \draw (-3+0.1,1.02) node[below left]{$-\vec w$};
    \draw (-2+0.2,1) node[above]{$\vec v-\vec w$};
    \draw (1-0.2,1) node[above]{$\vec u-\vec v$};
    \draw (2-0.2,-1) node[below]{$\vec w-\vec v$};

    \draw (4.1,1.2) node[above]{$\cK_1$};
    \draw (2,2) node[left]{$\cK_2$};

    \draw (-0.5,1) node[below]{edge};
    \draw (2.5,0) node[below]{sprout};
  \end{tikzpicture}
  \caption{Edge--sprout correspondence} \label{fig:edge_vs_sprout}
\end{figure}

Let us denote by $\cV(\cK_1)$ and $\cV(\cK_2)$ the sets of vertices of $\cK_1$ and $\cK_2$ respectively. Then, due to \eqref{eq:x_z_q_are_same}, \eqref{eq:omega_vs_mu}, and the correspondence described above, the relation \eqref{eq:mu_vs_partial_quotients} can be rewritten as
%
%
\begin{equation} \label{eq:omega_vs_integer_angle}
  \omega(\La_{\theta_1,\theta_2})=
  \tfrac12\hskip-2mm\displaystyle\limsup_{\substack{|\vec v|\to\infty \\ \vec v\in \cV(\cK_1)\cup \cV(\cK_2) }}\hskip-2mm\frac{\log(\alpha(\vec v))}{\log|\vec v|}.
\end{equation}
In other words, Diophantine exponents of lattices are responsible for the growth of integer angles at vertices of Klein polygons.

This point of view proposes also the following reformulation of Proposition \ref{prop:norm_minimum_vs_badly_approximable}.

\begin{proposition} \label{prop:norm_minimum_vs_integer_angles}
  \[N(\La_{\theta_1,\theta_2})>0\iff
    \sup_{\vec v\in \cV(\cK_1)\cup \cV(\cK_2)}\hskip-1mm\alpha(\vec v)<\infty.\]
\end{proposition}

\section{Spectra $\pmb{\Omega_d}$ and $\pmb{\tilde\Omega_d}$} \label{sec:spectra}

\subsection{Spectrum of $\pmb{\omega(\La)}$}

We remind that $\cL_d$ denotes the space of full rank lattices in $\R^d$ of covolume $1$. One of the first questions concerning lattice exponents is what values this quantity can attain. It follows from the definition of $\Pi(\vec x)$ that for each positive $t$ we have $\omega(t\La)=\omega(\La)$. Thus, all possible values of $\omega(\La)$ are provided by $\cL_d$, so, we can define the corresponding spectrum as
\[\Omega_d=\Big\{\omega(\Lambda)\,\Big|\,\Lambda\in\cL_d \Big\}.\]
As it was shown in the Introduction, $\Omega_d\subset[0,\infty]$. It is very natural to expect and challenging to prove
that, not only for $d=2$, but in any dimension, every nonnegative value is attainable by lattice exponents. 

\begin{problem} \label{pr:ray}
  Prove that $\Omega_d=[0,\infty]$.
\end{problem}

Theorem \ref{t:schmidt_finitely_many_points} can be applied to prove the existence of lattices with certain positive values of $\omega(\La)$. This way it was proved in \cite{german_2017} that the (finite) set
\begin{equation} \label{eq:spectrum_from_schmidt}
  \Big\{\,
  \frac{\,ab\,}{cd}\
  \Big|
  \begin{array}{l}
    a,b,c\in\N \\
    a+b+c=d
  \end{array}
  \Big\}
\end{equation}
is contained in $\Omega_d$.

Recently, it was shown in \cite{german_2018} that at least starting with some positive boundary every real number is contained in $\Omega_d$.

\begin{theorem}[O.G., 2018] \label{t:rays}
  For each $d\geq3$
  \[\bigg[3-\frac{d}{(d-1)^2}\,,\,+\infty\bigg]\subset\Omega_d\,.\]
\end{theorem}

It is also natural to seek a corresponding analogue of the Jarn\'ik--Besicovitch theorem, i.e. a multidimensional generalization of Theorem \ref{t:jarnik_besicovitch_for_lattices}.

\begin{problem} \label{pr:jarnik_besicovitch}
  Calculate or estimate
  \[\dim_\textup{H}\Big\{\Lambda\in\mathcal L_d\,\Big|\,\omega(\Lambda)\geq\lambda \Big\}.\]
  Prove that, as a function of $\lambda$, it is strictly decreasing for $\lambda\geq0$.
\end{problem}

Obviously, a positive solution to Problem \ref{pr:jarnik_besicovitch} implies the statement of Problem \ref{pr:ray}.

\subsection{Combined spectrum}

For each full rank lattice $\La$ let us denote by $\La^\ast$ the \emph{dual} lattice,
\[ \La^\ast=\Big\{ \vec y\in\R^d \,\Big|\, \langle\vec y,\vec x\rangle\in\Z\text{ for each }\vec x\in\La \Big\}, \]
where $\langle\,\cdot\,,\,\cdot\,\rangle$ is the inner product. It appears that, same as in many other Diophantine approximation settings, the phenomenon of \emph{transference} can be observed. This phenomenon connects dual problems. In the current setting those are the problems concerning $\La$ and $\La^\ast$, in particular $\omega(\La)$ and $\omega(\La^\ast)$.

Of course, if $d=2$ then $\La^\ast$ coincides up to a homothety with $\La$ rotated by $\pi/2$, so, in the two-dimensional case we obviously have $\omega(\La)=\omega(\La^\ast)$. In \cite{german_2017} the following transference theorem was proved.

\begin{theorem}[O.G., 2016] \label{t:lattice_transference}
  For each $\La\in\cL_d$ we have
  \begin{equation} \label{eq:lattice_transference}
    \omega(\La)\geq\frac{\omega(\La^\ast)}{(d-1)^2+d(d-2)\omega(\La^\ast)}\,.
  \end{equation}
  Here we mean that if $\omega(\La^\ast)=\infty$, then $\omega(\La)\geq\dfrac{1}{d(d-2)}$\,.
\end{theorem}

Thus, the structure of the \emph{combined spectrum}
\[\tilde\Omega_d=\Big\{\big(\omega(\La),\omega(\La^\ast)\big)\,\Big|\,\Lambda\in\cL_d \Big\}\]
is expected to be more complicated than that of $\Omega_d$. At least we definitely have $\tilde\Omega_d\neq[0;+\infty]\times[0;+\infty]$.

Since $(\La^\ast)^\ast=\La$, along with \eqref{eq:lattice_transference} a symmetric inequality holds, the one with $\La$ and $\La^\ast$ interchanged. Therefore,
\begin{equation} \label{eq:transference_for_lattice_exponents}
  \tilde\Omega_d\subset
  \left\{\big(x,y)\in[0,+\infty]^2\,\middle|\,
  \begin{array}{l}
    x\geq\dfrac{y}{(d-1)^2+d(d-2)y} \\
    y\geq\dfrac{x^{\vphantom{|}}}{(d-1)^2+d(d-2)x}
  \end{array} \right\}.
\end{equation}
Particularly, \eqref{eq:transference_for_lattice_exponents} implies that
\begin{equation*} 
  \omega(\La)=0\iff\omega(\La^\ast)=0.
\end{equation*}
Notice that this equivalence is similar to
\begin{equation*} 
  N(\La)>0\iff N(\La^\ast)>0\,
\end{equation*}
(see \cite{skriganov_1998} or \cite{german_2005}).

\begin{problem} \label{pr:lattice_transference}
  Describe $\tilde\Omega_d$ for $d\geq3$. Is it true that
  \[\tilde\Omega_d=
    \left\{\big(x,y)\in[0,+\infty]^2\,\middle|\,
    \begin{array}{l}
      x\geq\dfrac{y}{(d-1)^2+d(d-2)y} \\
      y\geq\dfrac{x^{\vphantom{|}}}{(d-1)^2+d(d-2)x}
    \end{array} \right\}\ ?\]
\end{problem}

In the proof of Theorem \ref{t:rays} (see \cite{german_2018}) the dual lattice is neglected, so, the only nonzero pairs $\big(\omega(\La),\omega(\La^\ast)\big)$ currently known to the author are $(\omega,+\infty)$, where $\omega$ is of the form \eqref{eq:spectrum_from_schmidt}. Moreover, the corresponding examples described in \cite{german_2017} have a certain flaw, as in each of them the dual lattice has some nonzero points in the coordinate planes, so that the condition $\omega(\La^\ast)=+\infty$ is provided by a kind of degeneracy. It would be more interesting to construct lattices that are \emph{totally irrational}, i.e. such that neither the lattice, nor its dual contains nonzero points in the coordinate planes. In this context it is worth mentioning the following result by N.\,Technau and M.\,Widmer \cite{technau_widmer}.

\begin{theorem}[Technau, Widmer, 2016] \label{t:technau_widmer}
  Let ${f}:(0,+\infty)\to(0,1)$ be non-increasing. Then there is a totally irrational lattice $\La\in\cL_d$ such that
  \[
  \begin{aligned}
    & \Pi(\vec x)\gg|\vec x|^{-d}\ \ \text{ for nonzero }\vec x\in\La, \\
    & \Pi(\vec x)\leq{f}(|\vec x|)\,\ \text{ for infinitely many }\vec x\in\La^\ast.
  \end{aligned}
  \]
\end{theorem}

Clearly, Theorem \ref{t:technau_widmer} implies the existence of a totally irrational $\La$ such that
\[0\leq\omega(\La)\leq d,\qquad\omega(\La^\ast)=+\infty.\]
Notice that in view of \eqref{eq:transference_for_lattice_exponents} the inequality $0\leq\omega(\La)\leq d$ for such $\La$ can be substituted by
\[\frac1{d(d-2)}\leq\omega(\La)\leq d.\]

\subsection{Linear forms of a given Diophantine type} \label{sec:linear_forms}

The proof of Theorem \ref{t:rays} is based on an existence theorem concerning linear forms of a given Diophantine type. It is rather difficult to control the values of all the $d$ forms at once. Controlling the values of each one of them separately is much simpler.

Given $\gamma,\delta\in\R$,
\begin{equation} \label{eq:gamma_delta_geq_1}
  \gamma\geq\delta\geq 1,
\end{equation}
suppose we can construct a linear form $L$ and a sequence $(\vec z_k)\subset\Z^d\backslash\{\vec 0\}$, $|\vec z_k|\to\infty$, such that

  \textup{(i)} $|L(\vec z)|\cdot|\vec z|^{d-1}\asymp|\vec z|^{-d\gamma}$ for $\vec z\in(\vec z_k)$;

  \textup{(ii)} $|L(\vec z)|\cdot|\vec z|^{d-1}\gg|\vec z|^{-d(\gamma-\delta)}$ for $\vec z\in\Z^d\backslash\bigcup_k\Z\vec z_k$;

  \textup{(iii)} the set of accumulation points of the sequence $(\vec z_k/|\vec z_k|)$ is not too large, for instance, consists of finitely many points.

Take arbitrary linear forms $L_1,\ldots,L_d$ such that none of them is zero at the accumulation points of $(\vec z_k/|\vec z_k|)$ and
\begin{equation} \label{eq:omega_leq_delta-1}
  \omega(\La)\leq\delta-1
\end{equation}
for $\La$ defined by \eqref{eq:lattice} (for instance, $\omega(\La)=0$). Then
\[
\begin{aligned}
  & |L_i(\vec z)|\asymp|\vec z|,\quad\,  i=1,\ldots,d,\quad\text{ for }\vec z\in(\vec z_k), \\
  & |L_i(\vec z)|\ll|\vec z|,\quad  i=1,\ldots,d,\quad\text{ for }\vec z\in\Z^d,
\end{aligned}
\]
and
\[
\begin{aligned}
  & |L_1(\vec z)\ldots L_{d-1}(\vec z)L(\vec z)|\asymp|\vec z|^{-d\gamma}\ \text{ for }\vec z\in(\vec z_k), \\
  & |L_1(\vec z)\ldots L_{d-1}(\vec z)L(\vec z)|=
    |L_1(\vec z)\ldots L_d(\vec z)|\frac{|L(\vec z)|\cdot|\vec z|^{d-1}}{|L_d(\vec z)|\cdot|\vec z|^{d-1}}\gg \\
  & \phantom{|L_1(\vec z)\ldots L_{d-1}(\vec z)L(\vec z)|}\gg
    |\vec z|^{-d(\delta-1)-d(\gamma-\delta)-d}=
    |\vec z|^{-d\gamma}
    \ \text{ for }
    \vec z\in\Z^d\backslash\textstyle\bigcup_k\Z\vec z_k.
\end{aligned}
\]
Hence
\[\omega(\La')=\gamma,\]
where
\[\La'=\Big\{\big(L_1(\vec z),\ldots,L_{d-1}(\vec z),L(\vec z)\big)\,\Big|\,\vec z\in\Z^d \Big\}.\]

Notice that the construction described cannot provide values of the exponent smaller than $1$. Nevertheless, Theorem \ref{t:rays} does not reach even this bound.

\begin{problem} \label{pr:linear_form}
  Given $\gamma\geq 0$, prove that there is a linear form $L$ and a sequence $(\vec z_k)\subset\Z^d\backslash\{\vec 0\}$, $|\vec z_k|\to\infty$, such that

  \textup{(i)} $|L(\vec z)|\cdot|\vec z|^{d-1}\asymp|\vec z|^{-d\gamma}$ for $\vec z\in(\vec z_k)$;

  \textup{(ii)} $|L(\vec z)|\cdot|\vec z|^{d-1}\gg1$ for $\vec z\in\Z^d\backslash\bigcup_k\Z\vec z_k$;

  \textup{(iii)} the set of accumulation points of the sequence $(\vec z_k/|\vec z_k|)$ is separated from some $(d-1)$-dimensional subspace of $\R^d$.
\end{problem}

In order to apply the statement of Problem \ref{pr:linear_form} as described above, we need $\gamma\geq1$ so that \eqref{eq:gamma_delta_geq_1} and \eqref{eq:omega_leq_delta-1} are consistent. Then it would give $[1,+\infty]\subset\Omega_d$, which is weaker than the statement of Problem \ref{pr:ray}, but Problem \ref{pr:linear_form} seems to be of independent interest.

As Nikolay Moshchevitin noticed, the existence of a linear form and a sequence of integer points satisfying statements \textup{(i)}, \textup{(ii)} of Problem \ref{pr:linear_form} should follow from Schmidt--Summerer's parametric geometry of numbers (see \cite{schmidt_summerer_2009}, \cite{schmidt_summerer_2013}, \cite{roy_annals_2015}) due to Roy's theorem (see \cite{roy_annals_2015}). The reason for this lies in the difference between best approximation vectors and all the other points. If $\vec z\in\Z^d$ is not an integer multiple of a best approximation vector for $L$, there is a point $\vec z'\in\Z^d$ linearly independent with $\vec z$ such that
\[\begin{cases}
  |L(\vec z')|\leq|L(\vec z)| \\
  |\underline{\vec z'}|\leq|\underline{\vec z}|
\end{cases},\]
where the underscore means the orthogonal projection to the hyperplane of the first $d-1$ coordinates.
Then
\[\lambda_2\bigg(\cC\Big(|\underline{\vec z}|\big/|L(\vec z)|\Big)\bigg)\leq|\underline{\vec z}|,\]
where
\[\cC(Q)=\Big\{\vec x\in\R^d\,\Big|\,|\underline{\vec x}|\leq1,\ |L(\vec x)|\leq Q^{-1} \Big\}\]
and $\lambda_j\big(\cC(Q)\big)$ is the $j$-th successive minimum of $\cC(Q)$ w.r.t. $\Z^d$. Thus, any lower bound for $\lambda_2\big(\cC(Q)\big)$ gives a lower bound in the spirit of statement \textup{(ii)} of Problem \ref{pr:linear_form}. However, Roy's theorem does not seem to immediately give any information concerning accumulation points of the set
\[\bigg\{\frac{\vec z}{|\vec z|}\,\bigg|\,\vec z\text{ is a best approximation vector for }L \bigg\},\]
which we rely upon when obtaining a bound for the lattice exponent.

\section{Multidimensional continued fractions} \label{sec:multicontractions}

\subsection{Facets and edge stars of Klein polyhedra} \label{sec:facets_and_edge_stars}

As it was mentioned in Section \ref{sec:klein_polygons}, in the two-dimensional case Diophantine exponents of lattices are responsible for the growth of integer angles at vertices of Klein polygons. It is natural to ask whether this connection can be generalized to arbitrary dimension.

Let $\La\in\cL_d$. Suppose $\La$ has no nonzero points in the coordinate planes. Let $\cO$ be one of the $2^d$ orthants, and let us consider the convex hull
\[\cK=\conv(\cO\cap\La\backslash\{\vec 0\}).\]
By analogy with the two-dimensional case, this convex hull is named \emph{Klein polyhedron}. As we assume $\La$ not to have any nonzero points in the coordinate planes, $\cK$ is a generalized polyhedron, i.e. its intersection with any bounded polyhedron is itself a polyhedron (see \cite{mussafir_2003}). Hence its boundary $\partial\cK$ has nice polyhedral structure, each vertex of $\cK$ is a lattice point incident to finitely many edges of $\cK$.

In the two-dimensional case the role of partial quotients is played by edges and pairs of adjacent edges of a Klein polygon. It is natural to expect that in the multidimensional case the same role is played by some local objects such as faces of different dimensions or by their unions. Besides that, we can consider some quantitative characteristics of those objects similar to integer lengths and integer angles.

Following this idea, in \cite{german_2005}, \cite{german_2007}, \cite{german_lakshtanov_2008} facets (faces of dimension $d-1$) and edge stars (unions of the edges incident to a vertex) are considered. 

Given a facet $F$ of $\cK$ with vertices $\vec v_1,\ldots,\vec v_k$, its \emph{determinant} is defined as
\[\det F=\sum_{1\leq i_1<\ldots<i_d\leq k}|\det(\vec v_{i_1},\ldots,\vec v_{i_d})|.\]

Given a vertex $\vec v$ of $\cK$, let $\vec r_1,\ldots,\vec r_k$ be the primitive lattice vectors parallel to the edges incident to $\vec v$. Denote by $\starv$ the edge star of $\vec v$. Then its \emph{determinant} is defined as
\[\det\starv=\sum_{1\leq i_1<\ldots<i_d\leq k}|\det(\vec r_{i_1},\ldots,\vec r_{i_d})|.\]

These quantities also equal the volumes of the Minkowski sums of $\vec v_1,\ldots,\vec v_k$ and of $\vec r_1,\ldots,\vec r_k$ respectively. They allow formulating a multidimensional generalization of Propositions \ref{prop:norm_minimum_vs_badly_approximable} and \ref{prop:norm_minimum_vs_integer_angles}. The following statement was proved in \cite{german_2005}, \cite{german_2007}.

\begin{theorem}[O.G., 2007] \label{t:german_norm_minima}
  Let $\cK_1,\ldots,\cK_{2^d}$ be the Klein polyhedra corresponding to $\La$ and all the $2^d$ orthants. Let $\cV(\cK_i)$ and $\cF(\cK_i)$ denote respectively the set of vertices and the set of facets of $\cK_i$. Then
  \begin{multline*}
    \hskip14mm
    N(\La)>0\iff
    \sup_{F\in\bigcup_i\cF(\cK_i)}\hskip-1mm\det F<\infty
    \hskip1.5mm
    \iff \\ \iff
    \sup_{\vec v\in\bigcup_i\cV(\cK_i)}\hskip-1mm\det\starv<\infty\iff
    \begin{cases}
      \displaystyle
      \sup_{F\in\cF(\cK_1)}\hskip-1mm\det F<\infty \\
      \displaystyle
      \sup_{\vec v\in\cV(\cK_1)}\hskip-1mm\det\starv<\infty
    \end{cases}.
    \hskip10mm
  \end{multline*}
\end{theorem}

It is reasonable to expect that this approach can be fruitful for generalizing not only Propositions \ref{prop:norm_minimum_vs_badly_approximable} and \ref{prop:norm_minimum_vs_integer_angles}, but also relation \eqref{eq:omega_vs_integer_angle}.

\begin{problem}
  Is it true that
  \begin{equation} \label{eq:omega_vs_det_starv}
    \omega(\La)\asymp
    \hskip-1mm\limsup_{\substack{|\vec v|\to\infty \\ \vec v\in\bigcup_i^{\vphantom|}\cV(\cK_i) }}\hskip-2mm\frac{\log(\det\starv)}{\log|\vec v|}\ ?
  \end{equation}
\end{problem}

Of course, instead of $\det\starv$ one can consider any other local integer linear or affine invariant, should it seem more appropriate.

\subsection{An argument in favour of \eqref{eq:omega_vs_det_starv}}

Let $\La$ and $\cK$ be as in \ref{sec:facets_and_edge_stars}. Let $|\cdot|$, as before, denote the sup-norm. Let $\vec v$ be a vertex of $\cK$, $\vec v=(v_1,\ldots,v_d)$. Set
\[D=\textup{diag}\Big(\Pi(\vec v)\big/|v_1|,\ldots,\Pi(\vec v)\big/|v_d|\Big).\]
Then $\cK'=D\cK$ is one of the $2^d$ Klein polyhedra corresponding to $\La'=D\La$, $\vec v'=D\vec v$ is its vertex, and
\[\det\textup{St}_{\vec v'}=\det\starv.\]
Moreover, $\vec v'$ is the shortest nonzero vector of $\La'$ and
\[|\vec v'|=\Pi(\vec v')=\Pi(\vec v).\]

Suppose we can choose vectors $\vec r_1,\ldots,\vec r_d$ among the primitive vectors of $\La'$ that are parallel to the edges of $\cK'$ incident to $\vec v'$ so that they satisfy

\vskip3mm
\textup{(i)} $|\vec r_1\wedge\ldots\wedge\vec r_i|\gg|\vec r_1\wedge\ldots\wedge\vec r_{i-1}|\cdot|\vec r_i|$ for each $i=2,\ldots,d$;

\textup{(ii)} any $d$ vectors among $\vec v'$, $\vec r_1,\ldots,\vec r_d$ are linearly independent.

\vskip3mm
\noindent
Then, on the one hand,
\begin{equation} \label{eq:prod_is_leq}
  \prod_{1\leq i\leq d}|\vec r_i|\asymp
  |\det(\vec r_1,\ldots,\vec r_d)|\leq
  \det\textup{St}_{\vec v'}=\det\starv.
\end{equation}
On the other,
\[\prod_{1\leq i\leq d}
  \bigg(|\vec v'|
        \prod_{\substack{0\leq j\leq d \\ j\neq i}}
        |\vec r_j|\bigg)\gg
  (\det\La')^d=1,\]
whence
\begin{equation} \label{eq:prod_is_geq}
  \prod_{1\leq i\leq d}|\vec r_i|\gg
  |\vec v'|^{-\frac d{d-1}}=
  \Pi(\vec v)^{-\frac d{d-1}}.
\end{equation}
Combining \eqref{eq:prod_is_leq} and \eqref{eq:prod_is_geq} we get
\begin{equation} \label{eq:Pi_leq_starv}
  \Pi(\vec v)^{-\frac d{d-1}}\ll\det\starv.
\end{equation}

Suppose, as before, $\cK_1,\ldots,\cK_{2^d}$ are all the $2^d$ Klein polyhedra of $\La$, and $\cV(\cK_i)$ is the set of vertices of $\cK_i$. Since the vertices of a Klein polyhedron give local minima of $\Pi(\vec x)$, \eqref{eq:omega_as_limsup} can be rewritten as
\begin{equation*} 
  \omega(\La)=
  \hskip-1mm\limsup_{\substack{|\vec v|\to\infty \\ \vec v\in\bigcup_i^{\vphantom|}\cV(\cK_i) }}\hskip-2mm\frac{\log\big(\Pi(\vec v)^{-1}\big)}{\log|\vec v|}\,.
\end{equation*}
Thus, \eqref{eq:Pi_leq_starv} implies
\[\omega(\La)\leq
  \frac{d-1}{d}
  \hskip-1mm\limsup_{\substack{|\vec v|\to\infty \\ \vec v\in\bigcup_i^{\vphantom|}\cV(\cK_i) }}\hskip-2mm\frac{\log(\det\starv)}{\log|\vec v|}\,,\]
which is a ``half'' of \eqref{eq:omega_vs_det_starv}.

Of course, statements \textup{(i)}, \textup{(ii)} are still to be properly proved, but they seem to be quite natural.

\section{Multiplicative parametric geometry of numbers} \label{sec:parametric}

\subsection{Successive minima} 

In the spirit of fundamental works \cite{schmidt_summerer_2009}, \cite{schmidt_summerer_2013}, \cite{roy_annals_2015} it is natural to propose the following approach. Most of the argument proposed is a translation to the current context of the argument of Schmidt and Summerer \cite{schmidt_summerer_2009}, and also of the paper \cite{german_AA_2012}. We remind that $|\cdot|$ denotes the sup-norm.

Let $\La\in\cL_d$. Set
\[\cB=\Big\{ \vec x\in\R^d\, \Big|\,|\vec x|\leq1 \Big\}.\]
For each $\pmb\tau=(\tau_1,\ldots,\tau_d)\in\R^d$, $\tau_1+\ldots+\tau_d=0$, set
\[D_{\pmb\tau}=\textup{diag}(e^{\tau_1},\ldots,e^{\tau_d})\]
and
\[\cB_{\pmb\tau}=D_{\pmb\tau}\cB.\]
Set also
\[|\pmb\tau|_+=\max\Big\{\tau_i\,\Big|\,\tau_i\geq0\Big\},
  \qquad
  |\pmb\tau|_-=|-\pmb\tau|_+=\max\Big\{|\tau_i|\,\Big|\,\tau_i\leq0\Big\}.\]
Clearly,
\[|\pmb\tau|=\max(|\pmb\tau|_-,|\pmb\tau|_+),\]
\[|\pmb\tau|_+/(d-1)\leq|\pmb\tau|_-\leq(d-1)|\pmb\tau|_+\ .\]

Let $\lambda_i(\cB_{\pmb\tau})$, $i=1,\ldots,d$, denote the $i$-th successive minimum, i.e.
the infimum of positive $\lambda$ such that $\lambda\cB_{\pmb\tau}$ contains at least $i$ linearly independent vectors of $\La$.
Finally, for each $i=1,\ldots,d$, let us set
\[L_i(\pmb\tau)=\log\big(\lambda_i(\cB_{\pmb\tau})\big).\]

\begin{proposition} \label{prop:properties_of_L_i}
  The functions $L_i(\pmb\tau)$ enjoy the following properties:

  \textup{(i)} $-|\pmb\tau|_++O(1)\leq L_1(\pmb\tau)\leq\ldots\leq L_d(\pmb\tau)\leq|\pmb\tau|_-+O(1)$; \vphantom{\bigg|}
  
  \textup{(ii)} $-\log d!\leq\sum_{1\leq i\leq d}L_i(\pmb\tau)\leq0$;

  \textup{(iii)} $L_i(\pmb\tau)$ is continuous and piecewise linear. \vphantom{\bigg|}
\end{proposition}

\begin{proof}
The inequalities $L_1(\pmb\tau)\leq\ldots\leq L_d(\pmb\tau)$ 
follow immediately from the definition of successive minima. The leftmost and the rightmost inequalities follow from the inclusions
\[e^{-|\pmb\tau|_+}\cB_{\pmb\tau}
  \subset\cB\subset
  e^{|\pmb\tau|_-}\cB_{\pmb\tau}.\]

Statement \textup{(ii)} follows from Minkowski's second theorem, which states that
\[\frac1{d!}\leq\prod_{1\leq i\leq d}\lambda_i(\cB_{\pmb\tau})\leq1.\]

Let us prove \textup{(iii)}. For each nonzero $\vec v\in\La$ let us denote by $\lambda_{\vec v}(\cB_{\pmb\tau})$ the infimum of positive $\lambda$ such that $\lambda\cB_{\pmb\tau}$ contains $\vec v$, and set
\[L_{\vec v}(\pmb\tau)=\log\big(\lambda_{\vec v}(\cB_{\pmb\tau})\big).\]
If $\vec v=(v_1,\ldots,v_d)$, then
\[\lambda_{\vec v}(\cB_{\pmb\tau})=\max_{1\leq i\leq d}(|v_i|e^{-\tau_i}),\]
and
\[L_{\vec v}(\pmb\tau)=\max_{1\leq i\leq d}\big(\log|v_i|-\tau_i\big),\]
i.e. $L_{\vec v}(\pmb\tau)$ is continuous and piecewise linear. Notice that for each $\pmb\tau$ and each $i=1,\ldots,d$ there is a $\vec v=\vec v(\pmb\tau,i)\in\La$ such that $\lambda_i(\cB_{\pmb\tau})=\lambda_{\vec v}(\cB_{\pmb\tau})$. Hence, denoting
\[\La_i=\Big\{ \vec v\in\La\, \Big|\, \exists\,\pmb\tau:\,\lambda_i(\cB_{\pmb\tau})\leq\lambda_{\vec v}(\cB_{\pmb\tau}) \Big\},\]
we get
\[L_i(\pmb\tau)=\min_{\vec v\in\La_i}L_{\vec v}(\pmb\tau).\]
Thus, $L_i(\pmb\tau)$ is indeed continuous and piecewise linear.
\end{proof}

\subsection{Schmidt--Summerer exponents} 


\begin{definition}
  We define the \emph{Schmidt--Summerer lower and upper exponents of the first type} as
  \[\bpsi_i(\La)=\liminf_{|\pmb\tau|\to\infty}\frac{L_i(\pmb\tau)}{|\pmb\tau|_+},
    \qquad
    \apsi_i(\La)=\limsup_{|\pmb\tau|\to\infty}\frac{L_i(\pmb\tau)}{|\pmb\tau|_+},\]
  and \emph{of the second type} as
  \[\bPsi_i(\La)=\liminf_{|\pmb\tau|\to\infty}\sum_{1\leq j\leq i}\frac{L_j(\pmb\tau)}{|\pmb\tau|_+},
    \qquad
    \aPsi_i(\La)=\limsup_{|\pmb\tau|\to\infty}\sum_{1\leq j\leq i}\frac{L_j(\pmb\tau)}{|\pmb\tau|_+}.\]
\end{definition}

The following inequalities are a straightforward corollary of Proposition \ref{prop:properties_of_L_i}:
\[
\begin{array}{c}
  -1\leq\bpsi_1(\La)\leq\ldots\leq\bpsi_d(\La)\leq d-1,\vphantom{} \\ 
  -1\leq\apsi_1(\La)\leq\ldots\leq\apsi_d(\La)\leq d-1,\vphantom{\bigg|} \\ \ \ \,
  \bPsi_d(\La)=\aPsi_d(\La)=0.
\end{array}
\]

Schmidt--Summerer exponents of lattices are, in a sense, global characteristics, whereas we could consider a one-parametric path $\big\{\pmb\tau(s)\,\big|\,s\in\R_+\big\}$ and the corresponding $\liminf$'s and $\limsup$'s as $s\to\infty$. This is performed in \cite{german_AA_2012} for the path defined by
\[\tau_1(s)=\ldots=\tau_m(s)=s,\quad\tau_{m+1}(s)=\ldots=\tau_d(s)=-ms/n,\]
which corresponds to the problem of simultaneous approximation of zero with the values of $n$ linear forms in $m$ variables, $n+m=d$. In that case Schmidt--Summerer exponents can be expressed in terms of \emph{intermediate Diophantine exponents} (see \cite{german_AA_2012}). In the current setting we have a similar situation: the exponents $\bpsi_1(\La)$ and $\omega(\La)$ are but two different points of view at the same phenomenon.

\begin{proposition} \label{prop:omega_vs_psi}
  $\omega(\La)^{-1}+\bpsi_1(\La)^{-1}+1=0$.
\end{proposition}

\begin{proof}
  For each $\vec v\in\La$ let us set
  \[\pmb\tau(\vec v)=\Big(\log\big(|v_1|\big/\Pi(\vec v)\big),\ldots,\log\big(|v_d|\big/\Pi(\vec v)\big)\Big).\]
  Then 
  \[
  \begin{array}{l}
    \lambda_{\vec v}(\cB_{\pmb\tau(\vec v)})=\Pi(\vec v), \\ \vphantom{\bigg|}
    L_{\vec v}(\pmb\tau(\vec v))=\log(\Pi(\vec v)), \\
    |\pmb\tau(\vec v)|_+=\log|\vec v|-\log(\Pi(\vec v)).
  \end{array}
  \]
  Hence
  \begin{multline*}
    \bpsi_1(\La)=
    \liminf_{|\pmb\tau|\to\infty}\frac{L_1(\pmb\tau)}{|\pmb\tau|_+}=
    \liminf_{|\pmb\tau|\to\infty}\frac{\min_{\vec v\in\La_1}L_{\vec v}(\pmb\tau)}{|\pmb\tau|_+}=
    \liminf_{|\pmb\tau|\to\infty}\frac{\min_{\vec v\in\La}L_{\vec v}(\pmb\tau)}{|\pmb\tau|_+}=\vphantom{\Bigg|} \\ =
    \liminf_{\substack{\vec v\in\La \\ |\vec v|\to\infty}}\frac{L_{\vec v}(\pmb\tau(\vec v))}{|\pmb\tau(\vec v)|_+}=
    \liminf_{\substack{\vec v\in\La \\ |\vec v|\to\infty}}\frac{\log(\Pi(\vec v))}{\log|\vec v|-\log(\Pi(\vec v))}= \\ =
    -\Bigg(1+\Bigg(\limsup_{\substack{\vec v\in\La \\ |\vec v|\to\infty}}\frac{\log\big(\Pi(\vec v)^{-1}\big)}{\log|\vec v|}\Bigg)^{-1}\Bigg)^{-1}=
    -\Big(1+\omega(\La)^{-1}\Big)^{-1}.
  \end{multline*}
\end{proof}

Proposition \ref{prop:omega_vs_psi} allows reformulating statements concerning Diophantine exponents of lattices in terms of Schmidt--Summerer exponents. Let us reformulate Problems \ref{pr:ray}, \ref{pr:lattice_transference} and Theorem \ref{t:lattice_transference}.

\begin{problem}[Reformulation of Problem \ref{pr:ray}] \label{pr:ray_schmimmerered}
  Prove that $\bpsi_1(\La)$ can attain any value in the interval $[-1,0]$.
\end{problem}

\begin{theorem}[Reformulation of Theorem \ref{t:lattice_transference}] \label{t:lattice_transference_schmimmerered}
  \begin{equation} \label{eq:lattice_transference_schmimmerered}
    \bpsi_1(\La)\leq\frac{\bpsi_1(\La^\ast)}{(d-1)^2}.
  \end{equation}
\end{theorem}

\begin{problem}[Reformulation of Problem \ref{pr:lattice_transference}]
  Is it true that the set of all possible values of $\big(\bpsi_1(\La),\bpsi_1(\La^\ast)\big)$ coincides with
  \[\left\{\big(x,y)\in[-1,0]^2\,\middle|\,
    (d-1)^2x\leq y\leq\frac{x}{(d-1)^2}
    \right\}\ ?\]
\end{problem}

Inequality \eqref{eq:lattice_transference_schmimmerered} stimulates the following natural question in the spirit of the papers \cite{schmidt_annals_1967}, \cite{laurent_up_down}, \cite{german_AA_2012}.

\begin{problem} \label{pr:ray_schmimmerered_split}
  Split the inequality \eqref{eq:lattice_transference_schmimmerered} into a chain of inequalities between the $\bpsi_i(\La)$ or $\bPsi_i(\La)$.
\end{problem}

And of course, we cannot omit the following most challenging question.

\begin{problem} \label{pr:roy}
  Prove an analogue of Roy's theorem on rigid systems (see \cite{roy_annals_2015}) for the functions $L_1(\pmb\tau),\ldots,L_d(\pmb\tau)$.
\end{problem}

\paragraph{Acknowledgements.}

The author is grateful to La Trobe University and Sydney University, and personally to Mumtaz Hussain and Dzmitry Badziahin, for warm hospitality and fruitful discussions.

The author is also a Young Russian Mathematics award winner and would like to thank its sponsors and jury.


\begin{thebibliography}{99}

\bibitem
    {skriganov_1998}
    \textsc{M.\,M.\,Skriganov}
    \textit{Ergodic theory on $\SL(n)$, Diophantine approximations and anomalies in the lattice point problem.}
    Invent. math., \textbf{132} (1998), 1--72.
\bibitem
    {german_2017}
    \textsc{O.\,N.\,German}
    \textit{Diophantine exponents of lattices.}
    Proc. Steklov Inst. Math., \textbf{296}, suppl. 2 (2017), 29--35.
\bibitem
    {kleinbock_margulis_1999}
    \textsc{D.\,Y.\,Kleinbock, G.\,A.\,Margulis}
    \textit{Logarithm laws for flows on homogeneous spaces.}
    Invent. math., \textbf{138} (1999), 451--494.
\bibitem
    {cassels_GN}
    \textsc{J.\,W.\,S.\,Cassels}
    \textit{An introduction to the geometry of numbers.}
    Springer (1997).
\bibitem
    {borevich_shafarevich}
    \textsc{Z.\,I.\,Borevich, I.\,R.\,Shafarevich}
    \textit{Number theory}.
    NY Academic Press (1966).
\bibitem
    {margulis_2000}
    \textsc{G.\,A.\,Margulis}
    \textit{Problems and conjectures in rigidity theory.}
    Mathematics: frontiers and perspectives, Amer. Math. Soc., Providence, RI (2000), 161--174.
\bibitem
    {cassels_swinnerton_dyer}
    \textsc{J.\,W.\,S.\,Cassels, H.\,P.\,F.\,Swinnerton--Dyer}
    \textit{On the product of three homogeneous linear forms and indefinite ternary quadratic forms.}
    Phil. Trans. Royal Soc. London, \textbf{A 248} (1955), 73--96.
\bibitem
    {einsiedler_katok_lindenstrauss}
    \textsc{M.\,Einsiedler, A.\,Katok, E.\,Lindenstrauss}
    \textit{Invariant measures and the set of exceptions to Littlewood’s conjecture.}
    Ann. Math., \textbf{164}:2 (2006), 513--560.
\bibitem
    {karpenkov_book}
    \textsc{O.\,N.\,Karpenkov}
    \textit{Geometry of Continued Fractions.}
    Algorithms and Computation in Mathematics, \textbf{26}, Springer-Verlag (2013).
\bibitem
    {german_tlyustangelov}
    \textsc{O.\,N.\,German, I.\,A.\,Tlyustangelov}
    \textit{Palindromes and periodic continued fractions.}
    Moscow Journal of Combinatorics and Number Theory, \textbf{6}:2-3 (2016), 354--373.
\bibitem
    {german_2018}
    \textsc{O.\,N.\,German}
    \textit{Linear forms of a given Diophantine type and lattice exponents.}
    Preprint, arXiv:1804.01152.
\bibitem
    {german_2005}
    \textsc{O.\,N.\,German}
    \textit{Sails and norm minima of lattices.}
    Sbornik: Mathematics, \textbf{196}:3 (2005), 337--365.
\bibitem
    {technau_widmer}
    \textsc{N.\,Technau, M.\,Widmer}
    \textit{On a counting theorem of Skriganov.}
     Preprint, arXiv:1611.02649.
\bibitem
    {schmidt_summerer_2009}
    \textsc{W.\,M.\,Schmidt, L.\,Summerer}
    \textit{Parametric geometry of numbers and applications.}
    Acta Arithmetica \textbf{140} (2009), 67--91.
\bibitem
    {schmidt_summerer_2013}
    \textsc{W.\,M.\,Schmidt, L.\,Summerer}
    \textit{Diophantine approximation and parametric geometry of numbers.}
    Monatsh. Math. \textbf{169} (2013), 51--104.
\bibitem
    {roy_annals_2015}
    \textsc{D.\,Roy}
    \textit{On Schmidt and Summerer parametric geometry of numbers.}
    Ann. Math. \textbf{182}:2 (2015), 739--786.
\bibitem
    {mussafir_2003}
    \textsc{J.-O.\,Mussafir}
    \textit{Convex hulls of integral points.}
    J. of Math. Sciences, \textbf{113}:5 (2003), 647--665.
\bibitem
    {german_2007}
    \textsc{O.\,N.\,German}
    \textit{Klein polyhedra and lattices with positive norm minima},
    Journal de Th\'eorie des Nombres de Bordeaux, \textbf{19} (2007), 157--190.
\bibitem
    {german_lakshtanov_2008}
    \textsc{O.\,N.\,German, E.\,L.\,Lakshtanov}
    \textit{On a multidimensional generalization of Lagrange's theorem on continued fractions},
    Izvestiya: Mathematics, \textbf{72}:1 (2008), 47--61.
\bibitem
    {german_AA_2012}
    \textsc{O.\,N.\,German}
    \textit{Intermediate Diophantine exponents and parametric geometry of numbers},
    Acta Arithmetica, \textbf{154}:1 (2012), 79--101.
\bibitem
    {schmidt_annals_1967}
    \textsc{W.\,M.\,Schmidt}
    \textit{On heights of algebraic subspaces and diophantine approximations.}
    Ann. Math. \textbf{85}:3 (1967), 430--472.
\bibitem
    {laurent_up_down}
    \textsc{M.\,Laurent}
    \textit{On transfer inequalities in Diophantine Approximation.}
    ``Analytic Number Theory, Essays in Honour of Klaus Roth'' (ed. W.\,W.\,L.\,Chen, W.\,T.\, Gowers, H.\,Halberstam, W.\,M.\,Schmidt and R.\,C.\,Vaughan). Cambridge University Press (2009), 306--314.

\end{thebibliography}
\end{document}